\documentclass[12pt]{amsart}
\usepackage{amssymb,verbatim,amscd,amsmath,graphicx,enumerate}
\usepackage{amsmath}
\usepackage{xcolor}
\usepackage{hyperref}

\usepackage[margin=1 in]{geometry}

\usepackage{caption}
\usepackage{float}
\usepackage{pgf,tikz}
\usepackage{mathrsfs}
\usepackage{subcaption}
\usepackage{pdfsync}
\usepackage{fullpage}
\usepackage{color}
\usepackage{enumitem}
\usepackage{marginnote}
\usepackage{tikz}
\usepackage[style=numeric, giveninits=true, maxbibnames=9,doi=false,isbn=false,url=false,eprint=false]{biblatex}
\usetikzlibrary{patterns}
\usetikzlibrary{calc}
\usetikzlibrary{matrix}
\usepgflibrary{arrows}
\usetikzlibrary{arrows}
\usetikzlibrary{decorations.pathreplacing}
\usetikzlibrary{decorations.pathmorphing}

\numberwithin{equation}{section}
\newtheorem{theorem}{Theorem}[section]
\newtheorem{lemma}[theorem]{Lemma}
\newtheorem{proposition}[theorem]{Proposition}
\newtheorem{corollary}[theorem]{Corollary}
\newtheorem{conjecture}[theorem]{Conjecture}

\theoremstyle{definition}

\newtheorem{def-prop}[theorem]{Definition-Proposition}
\newtheorem{remark}[theorem]{Remark}

\newtheorem{question}[theorem]{Question}

\newtheorem*{Mysketch}{Sketch of proof} % \newtheorem establishes the object heading
  {\pushQED{\qed}\begin{Mysketch}}
  {\popQED\end{Mysketch}}

\DeclareMathOperator{\Ass}{Ass}

\DeclareMathOperator{\height}{ht}

\DeclareMathOperator{\pd}{pd}
\DeclareMathOperator{\bigheight}{bight}

\DeclareMathOperator{\vn}{v}

\def\KK{\mathbb{K}}

\begin{document}

\title[Binomial edge ideals of crown graphs]{Binomial edge ideals of crown graphs}

\author{Arvind Kumar, Joshua Pomeroy,  and Le Tran}
\address{Department of Mathematical Sciences \\
New Mexico State University\\
Las Cruces, NM 88003}
\email{arvkumar@nmsu.edu}
\urladdr{https://sites.google.com/view/arvkumar/home}
\email{reuel@nmsu.edu}
\email{letran95@nmsu.edu}
\urladdr{ https://sites.google.com/view/tran-ngocle}

\keywords{Binomial Edge Ideals, Crown Graphs, Krull dimension, Big-height, projective dimension, $\vn$-number}
\subjclass[2020]{13C10, 13C15,  05E40, 13F20}

\begin{abstract}
In this article, we explore the class of graphs for which the projective dimension of the quotient of the binomial edge ideals matches the big height of that ideal. Additionally, we investigate the Vasconcelos number of binomial edge ideals for cycles and crown graphs. We also provide proof for  \cite[Conjecture 4.13]{vconjecture} which is related to the Vasconcelos number of binomial edge ideals for cycles. \end{abstract}
\maketitle

\section{Introduction}\label{intro}
It is well-known that if $I$ is a homogeneous ideal in a standard graded polynomial ring $S$, then the projective dimension of the quotient $\frac{S}{I}$ is greater than or equal to the big height of $I$, where the \textit{big height} of an ideal is the largest height of the minimal primes of that ideal. When $\frac{S}{I}$ is Cohen-Macaulay, then the projective dimension of $\frac{S}{I}$ is the same as $\bigheight(I)$. Our goal is to understand which class of ideals, other than Cohen-Macaulay, has the same projective dimension as the big height of $I$. Specifically, we are exploring this question in the context of binomial edge ideals. In this article, we investigate this question for binomial edge ideals associated with various classes of graphs, which include block graphs, cycles, wheel graphs, complete multipartite graphs, crown graphs, etc.

Our other objective is to study the local Vasconcelos number of binomial edge ideals for graphs. Let $I$ be a proper non-zero homogeneous ideal in $S$, and let $\mathfrak{p} \in \Ass(I)$. The \textit{local Vasconcelos number} of $I$ with respect to $\mathfrak{p}$, denoted by $\vn_{\mathfrak{p}}(I)$, is the least possible degree of a homogeneous element $f$ such that $I:f=\mathfrak{p}$, i.e., $$\vn_{\mathfrak{p}}(I):= \min \{ d~:~ \exists f \in S_d \text{ so that } I:f=\mathfrak{p}\},$$ where $S_d$ is the $\KK$-vector space spanned by all the monomials of degree $d$.  The \textit{Vasconcelos number}, abbreviated as the \textit{$\vn$-number}, of $I$, denoted by $\vn(I)$, is the minimum of the local Vasconcelos numbers of $I$, i.e., $$\vn(I):= \min \{ \vn_\mathfrak{p}(I)~:~ \mathfrak{p}\in \Ass(I)\}.$$ 
The concept of the $\vn-$number of homogeneous ideals was first introduced by Cooper et al. \cite{CSTPV20}. In this article, we examine the $\vn$-number of binomial edge ideals of crown graphs and cycles. We prove a recent conjecture by Deblina et al. \cite{vconjecture} on the $\vn$-number of binomial edge ideals of cycles.

The binomial edge ideal of a graph was introduced by Herzog et al. in \cite{HHHKR10} and independently by Ohtani in \cite{Oh11}. Let $G$ be a simple graph on the vertex set $[n]:=\{1,\ldots,n\}.$ Let $S = \KK[x_1, \ldots, x_n,y_1,\ldots,y_n] $  be a standard graded polynomial ring in $2n$ variables over an arbitrary field \( \KK \). The \textit{binomial edge ideal} of $G$ is defined as $$J_G:=\langle x_iy_j-x_jy_i~:~ \{i,j\} \in E(G) \rangle.$$

The focus of this article is on crown graphs and cycles. We examine the binomial edge ideals of crown graphs with the intent of achieving our main objectives. To start, we analyze the structures of the minimal primes of the binomial edge ideals of crown graphs (see Section \ref{Krulldim}). We utilized the minimal primes description provided by Herzog et al. \cite{HHHKR10} to deduce the structure of the minimal primes of binomial edge ideals of crown graphs. And then using the structure of minimal primes, we determine the big height of the binomial edge ideal of crown graphs.

In Section \ref{depth}, our main objective is to determine for which class of graphs the projective dimension of the quotient of their binomial edge ideals is the same as the big height. We prove that for cycles, block graphs, wheel graphs, and most complete multipartite graphs, the projective dimension is the same as the big height. Additionally, we show that the binomial edge ideals of maximal projective dimension have projective dimension same as their big height. We compute the projective dimension of the binomial edge ideals of crown graphs and prove that it is the same as the big height.

In Section \ref{localvnumber}, we delve into the $\vn$-number of binomial edge ideals of crown graphs and cycles. We utilize Ohtani's recursive method from \cite{Oh11} to establish this conjecture. Ohtani demonstrated in \cite{Oh11} that for any graph $G$ and an internal vertex $v$ of $G$, the following holds: $$J_G=J_{G_v}\bigcap \left( ( x_v,y_v ) + J_{G\setminus v}\right).$$ This technical lemma of Ohtani was initially employed by the first author in \cite{K21} and in his thesis. Subsequently, the first author extended this technical lemma to the case of generalized binomial edge ideals in \cite{genbin}. Since then, this technical lemma has been utilized in numerous papers and has become a central tool. We employ this technical lemma to investigate the local $\vn$-number of binomial edge ideals of crown graphs. Furthermore, we use it to prove a conjecture by Deblina et al. in \cite{vconjecture} regarding the $\vn$-number of cycles.

{\bf Acknowledgement:} This research is conducted and completed during a focused five-week summer research program in Summer 2024, designed and supervised by the first author. The program is financially supported by the Department of Mathematical Sciences at New Mexico State University. We express our sincere gratitude to the department for supporting this research. We also gratefully acknowledge the anonymous referee for their insightful comments and constructive suggestions, which greatly improve the clarity and quality of this manuscript.

\section{Preliminaries}\label{preliminaries}

In this section, we  recall  some notation and definitions on
graphs which we will use throughout the article.

Let $G$ be a finite simple graph with a vertex set $V(G)$ and an edge set $E(G)$. When we refer to $G[A]$, we mean the {\it induced subgraph} of $G$ on the vertex set $A$. In other words, for any two vertices $i, j \in A$, the edge $\{i,j\}$ is in $E(G[A])$ if and only if $\{i,j\}$ is in $E(G)$. 
If we take a vertex $v$, $G \setminus v$ refers to the induced subgraph of $G$ on the vertex set $V(G) \setminus \{v\}$. Meanwhile, $N_G(v) := \{u \in V(G) : \{u,v\} \in E(G)\}$ denotes the {\it neighborhood} of $v$. Additionally, $G_v$ is the graph on the vertex set $V(G)$ with edge set $E(G_v) = E(G) \cup \{ \{u,w\} : u,w \in N_G(v)\}$. The {\it degree} of a vertex $v$, represented by $\deg_G(v)$, is the size of $N_G(v)$.

A \textit{complete graph} is one in which every vertex is adjacent to every other vertex. A complete graph on $n$ vertices is denoted by $K_n$. A graph is said to be {\it connected} if, for every pair of distinct vertices, there exists a sequence of edges of the graph connecting them. If a graph is not connected, it is said to be {\it disconnected}. A vertex $v$ of $G$ is called a {\it cut vertex} if the number of connected components of $G \setminus v$ is more than the number of connected components of $G$. A maximal connected subgraph of \(G\) with no cut vertex is called a \textit{block}. A graph \(G\) is called a {\it block graph} if each block of \(G\) is a complete graph.

A vertex $v\in V(G)$ is called a \textit{free} or \textit{simplicial} vertex if $v$ belongs to exactly one maximal complete subgraph of $G$. Any vertex that belongs to more than one maximal complete subgraph of $G$ is called an \textit{internal} vertex. The \textit{path graph} on $n$ vertices, denoted by $P_n$, is a graph in which the vertices can be ordered as $\{v_1,\ldots, v_n\}$ such that the edge set is $\{\{v_i,v_{i+1}\}~:~ 1 \le i \le n-1\}$. The \textit{cycle graph} on $n$ vertices, denoted by $C_n$, is a graph in which the vertices can be ordered as $\{v_1,\ldots, v_n\}$ such that the edge set is $\{\{v_i,v_{i+1}\}~:~ 1 \le i \le n-1\}\cup \{\{v_1,v_n\}\}$. A \textit{bipartite graph} is a graph whose vertex set can be partitioned into two sets, such that every edge connects a vertex in one set to a vertex in the other set.

A {\it crown graph} is a bipartite graph denoted by \(C_{n,n}\), with vertex set \([2n]\) divided into two sets: \(X=\{2i-1~:~i \in [n]\}\) and \(Y=\{2i~:~i \in [n]\}\). This graph has edges \(\{2i-1, 2j\}\) whenever \(i,j \in [n]\) and \(i \neq j\). Figure \ref{fig:$C_{5,5}$} represents the crown graph, $C_{5,5}$, on $10 $ vertices. 
 
\begin{figure}[ht]
    \centering
     \begin{tikzpicture}
        \filldraw (-1,4) circle(.3ex); 
         \filldraw (0,4) circle(.3ex); 
          \filldraw (1,4) circle(.3ex); 
           \filldraw (-2,4) circle(.3ex); 
            \filldraw (2,4) circle(.3ex); 
\filldraw (0,2) circle(.3ex);
\filldraw (-1,2) circle(.3ex);
\filldraw (-2,2) circle(.3ex);
\filldraw (1,2) circle(.3ex);
   \filldraw (2,2) circle(.3ex);
\draw (2,4)--(-1,2);
\draw (2,4)--(-2,2);
\draw (2,4)--(0,2);
\draw (2,4)--(1,2);
\draw (-2,4)--(-1,2);
\draw (0,4)--(-1,2);
\draw (1,4)--(-1,2);
\draw (-1,4)--(2,2);
\draw (-1,4)--(-2,2);
\draw (-1,4)--(0,2);
\draw (-1,4)--(1,2);
\draw (-2,4)--(2,2);
\draw (1,4)--(2,2);
\draw (0,4)--(2,2);
\draw (-2,4)--(0,2);
\draw (-2,4)--(1,2);
\draw (1,4)--(0,2);
\draw (1,4)--(-2,2);
\draw (0,4)--(1,2);
\draw (0,4)--(-2,2);
\node at (-1,4.3) {$3$};
\node at (-2,4.3) {$1$};
\node at (0,4.3) {$5$};
     \node at (1,4.3) {$7$};
     \node at (2,4.3) {$9$};
     \node at (-2,1.7) {$2$};
     \node at (-1,1.7) {$4$};
     \node at (0,1.7) {$6$};
     \node at (1,1.7) {$8$};
          \node at (2,1.7) {$10$};
          
         % \drawplot[smooth cycle] {(0,0),(1,1),(2,2)};
  \end{tikzpicture}
    \caption{$C_{5,5}$}
    \label{fig:$C_{5,5}$}
\end{figure}

A subset $T \subseteq V(G)$ is called a \textit{cut set} if the number of connected components of $G[\bar T]$ is greater than the number of connected components of $G$, where $\bar{T}:=V(G)\setminus T$. We say that $G$ is $k$-\textit{vertex-connected} if $k < |V(G)|$ and for every $A \subset [n]$ with $|A| < k$, the induced graph $G[\bar{A}]$ is connected. The \textit{vertex connectivity} of a connected graph $G$, denoted by $\kappa(G)$, is defined as the maximum positive integer $k$ such that $G$ is $k$ vertex-connected.

A {\it dominating set} of a connected graph is a subset $T \subseteq V(G)$ such that for every $v \in V(G)$, either $v \in T$ or $v$ is adjacent to a vertex in $T$. A {\it connected dominating set} is a dominating set for which the induced subgraph is connected. The {\it connected domination number} is the size of the smallest connected dominating set of $G$.

If we are given two graphs, $G_1$ and $G_2$, the {\it join product of $G_1$ and $G_2$}, represented as $G_1 \ast G_2$, is a graph with vertex set $V(G_1) \sqcup V(G_2)$, and edges from the combined set of edges from $E(G_1)$ and $E(G_2)$, along with additional edges connecting each vertex in $V(G_1)$ to every vertex in $V(G_2)$. A {\it complete $k$-partite} graph is a join product of $k$ empty graphs, where an {\it empty graph} is a graph with no edges.

\section{Krull dimension of binomial edge ideals}\label{Krulldim}
In this section, we study the Krull dimension and the considerable height of the binomial edge ideals of the crown graphs. To compute these invariants, we explicitly understand the structure of the associated primes of the binomial edge ideals of the crown graphs, which will be helpful in studying other related invariants later in the paper.

For a positive integer $n$, we set $[n]:=\{1,\ldots,n\}$.  Let $G$ be a simple graph with the vertex set $V(G)=[n]$ and edge set $E(G)$. Let 
$S=\KK[x_1, \ldots , x_n, y_1, \ldots , y_n]$ be a standard graded polynomial ring in $2n$ variables over a field $\KK$. The \textit{binomial edge ideal} of $G$ is defined as
$$J_G := \langle x_i y_j - x_j y_i~:~ \{i, j\} \in E(G) \rangle.$$ This class of ideals was introduced by Herzog et al. in \cite{HHHKR10} and independently by Ohtani in \cite{Oh11} and has been the subject of many graduate theses and research articles since then. 

Herzog et al. in \cite{HHHKR10} and  Ohtani in \cite{Oh11} proved that $J_G$ is a radical ideal. Herzog et al. in \cite{HHHKR10} obtained the complete description of minimal associated primes of $J_G$ while Ohtani in \cite{Oh11} provided a recursive way to obtain associated primes of $J_G$. Both of these descriptions have been heavily used in the literature. We use both of these descriptions for our research in this article. We recall here the description of minimal associated primes given by Herzog et al. in \cite{HHHKR10}, and we will return to the recursive method by Ohtani \cite{Oh11}  later in the article.

For  $T \subset [n]$, we set $\bar{T} := [n]\setminus T$ and $c_G(T)$ for the number of connected components of $G[\bar{T}]$, where $G[\bar T]$ is an induced subgraph of $G$ on the vertex set $\bar T$. Let $G_1,\ldots,G_{c_G(T)}$ be the connected components of $G[\bar{T}]$ with vertex sets $V(G_1),\ldots,V(G_{c_G(T)})$, respectively. For $1 \le i \le c_G(T)$, 
 we set $\tilde{G}_i$ for the complete graph on $V(G_i)$. Set 
$$P_T(G) := \langle x_i, y_i~:~ i \in T \rangle + J_{\tilde{G_1}}+\cdots +J_{\tilde{G}_{c_G(T)}}.$$ The  notion  $P_T(G)$ was introduced by 
Herzog et al. in \cite{HHHKR10} and they proved that $P_T(G)$ is a prime ideal containing $J_G$ and (see \cite[Theorem 3.2]{HHHKR10}) $$J_G =  \underset{T \subseteq [n]}\bigcap P_T(G).$$ 
However, for every $T \subseteq [n]$, $P_T(G)$ may not be a minimal prime of $J_G$. The minimal primes correspond to sets with a special combinatorial property known as \textit{cut point property}. A subset $T \subseteq [n]$ has the cut point property if either $T=\emptyset$ or for each $i \in T$,  $i$ is a cut vertex of the graph $G[\bar{T} \cup \{i\}]$, i.e., $c_G(T)>c_{G}(T \setminus \{i\})$.
We set $\mathcal{C}(G) := \{ T~:~ T \; \text{has the cut point property} \}.$ Then, it follows from Herzog et al. \cite[Corollary 3.9]{HHHKR10} 
that $P_T(G)$ is a minimal prime of $J_G$ if and only if $T \in \mathcal{C}(G)$.

To compute the Krull dimension, big height, and multiplicity of the binomial edge ideals of the crown graphs, we understand the structure of sets with the cut point property. We begin our investigation with a technical lemma.

\begin{lemma}\label{lem:tech-connected}
    Let $G= C_{n,n}$ with $n\ge 3$ and let $T\subseteq [2n]$. We set  $X := \{1,3,\ldots, 2n-1\}$ and $Y := \{2,4,\ldots, 2n\}$. 
Suppose $\vert X\setminus T\vert \ge 2$, $\vert Y \setminus T \vert \ge 2$, and $\vert X \setminus  T \vert + \vert Y \setminus T\vert \ge 5$. Then, $G[\bar T]$ is connected.
\end{lemma}

\begin{proof}
First, note that $G[\bar T]$ is a graph on the vertex set $\left(X\setminus T \right)\sqcup \left(Y \setminus T\right)$. Since $\vert X\setminus T\vert \ge 2$, $\vert Y\setminus T\vert \ge 2$, and $\vert X\setminus T \vert + \vert Y\setminus T\vert \ge 5$, either $|X\setminus T| \ge 3$ or $|Y\setminus T| \ge 3.$ Assume, without loss of generality, that $|X\setminus T| \ge 3.$ Then, there exist $ 1 \le i<j<k \le  n$ and $1 \le  \ell <  m \le n$ such that $2i-1,2j-1,2k-1\in X\setminus T$ and $2\ell, 2m \in Y\setminus T.$ We claim that the induced subgraph on $\{2i-1,2j-1,2k-1,2\ell, 2m\}$ is connected. We have the following cases: 
 \begin{minipage}{\linewidth}
  \begin{minipage}{0.6\linewidth}
(1) Suppose $\{i,j,k\} \cap \{\ell,m\} =\{\ell,m\}$. Without loss of generality, assume that $i=\ell$ and $j=m$. Then, $2k-1$ is adjacent to both $2\ell,2m$ and $2\ell$ is adjacent to both $2j-1,2k-1$, and $2m$ is adjacent to both $2i-1,2k-1$, see Figure \ref{fig:case1}. \end{minipage}\hspace{0.5in}
\begin{minipage}{0.35\linewidth}
    \begin{figure}[H]
    \centering
    \begin{minipage}{\linewidth}
     \begin{tikzpicture}[scale=0.65]
\captionsetup{justification=centering,margin=2cm}
        \filldraw (-2,4) circle(.3ex); 
  \filldraw (-2,2) circle(.3ex);
\filldraw (0,2) circle(.3ex);
\filldraw (2,4) circle(.3ex);
   \filldraw (0,4) circle(.3ex);
\draw (-2,4)--(0,2);
\draw (-2,2)--(2,4);
\draw (0,4)--(-2,2);
\draw (0,2)--(2,4);
    \node at (-2,4.3) {$2i-1$};
    \node at (-2,1.7) {$2\ell$};
      \node at (0,1.7) {$2m$};
        \node at (2,4.3) {$2k-1$};
          \node at (0,4.3) {$2j-1$};
  \end{tikzpicture}    
    \caption{} \label{fig:case1}
    \end{minipage}
\end{figure} 
\end{minipage}
\end{minipage}
 \begin{minipage}{\linewidth}
  \begin{minipage}{0.6\linewidth} (2) Suppose $\{i,j,k\} \cap \{\ell,m\} =\{\ell\}$.  In this situation, $2m$ is adjacent to $2i-1,2j-1,2k-1$, and $2\ell$ is adjacent to two vertices among $\{2i-1,2j-1,2k-1\},$ see Figure \ref{fig:case2}. \end{minipage}\hspace{0.5in}
\begin{minipage}{0.35\linewidth}
\begin{figure}[H] \centering
     \begin{minipage}{\linewidth}
     \begin{tikzpicture}[scale=0.65]
\captionsetup{justification=centering,margin=2cm}
\filldraw (-1,4) circle(.3ex); 
  \filldraw (-1,2) circle(.3ex);
\filldraw (1,2) circle(.3ex);
\filldraw (2,4) circle(.3ex);
   \filldraw (-3,4) circle(.3ex);
\draw (-1,4)--(1,2);
\draw (-1,2)--(2,4);
\draw (-3,4)--(-1,2);
\draw (-3,4)--(1,2);
\draw (1,2)--(2,4);
    \node at (-1,4.3) {$2j-1$};
    \node at (-1,1.7) {$2\ell $};
      \node at (1,1.7) {$2m$};
        \node at (2,4.3) {$2k-1$};
          \node at (-3,4.3) {$2i-1$};
  \end{tikzpicture}
    \caption{}\label{fig:case2}
\end{minipage}
\end{figure} 
\end{minipage}
\end{minipage}
 \begin{minipage}{\linewidth}
  \begin{minipage}{0.6\linewidth} (3) Suppose $\{i,j,k\} \cap \{\ell,m\} =\{m\}$. In this situation, $2\ell$ is adjacent to $2i-1,2j-1,2k-1$, and $2m$ is adjacent to two vertices among $\{2i-1,2j-1,2k-1\}$, see Figure \ref{fig:case3}. \end{minipage}\hspace{0.5in}
\begin{minipage}{0.35\linewidth} \begin{figure}[H] \centering
    \begin{minipage}{\linewidth}
     \begin{tikzpicture}[scale=0.65]
\captionsetup{justification=centering,margin=2cm}
        \filldraw (0,4) circle(.3ex); 
  \filldraw (-2,2) circle(.3ex);
\filldraw (0,2) circle(.3ex);
\filldraw (2,4) circle(.3ex);
   \filldraw (-3,4) circle(.3ex);
\draw (0,4)--(-2,2);
\draw (-2,2)--(2,4);
\draw (-3,4)--(-2,2);
\draw (-3,4)--(0,2);
\draw (0,2)--(2,4);
    \node at (0,4.3) {$2j-1$};
    \node at (-2,1.7) {$2\ell $};
      \node at (0,1.7) {$2m$};
        \node at (2,4.3) {$2k-1$};
          \node at (-3,4.3) {$2i-1$};
  \end{tikzpicture}
    \caption{}\label{fig:case3}
\end{minipage}
\end{figure} 
\end{minipage}
\end{minipage}
     \begin{minipage}{\linewidth}
  \begin{minipage}{0.6\linewidth} (4) Suppose $\{i,j,k\} \cap \{\ell,m\} =\emptyset.$ In this case, $2i-1,2j-1,2k-1$ are adjacent to both $2\ell,2m$, see Figure \ref{fig:case4}. 
    \end{minipage}\hspace{0.5in}
\begin{minipage}{0.35\linewidth} \begin{figure}[H] \centering
    \begin{minipage}{\linewidth}
     \begin{tikzpicture}[scale=0.65]
\captionsetup{justification=centering,margin=2cm}
      \filldraw (-1,4) circle(.3ex); 
  \filldraw (-2,2) circle(.3ex);
\filldraw (0,2) circle(.3ex);
\filldraw (1,4) circle(.3ex);
   \filldraw (-3,4) circle(.3ex);
\draw (-1,4)--(0,2);
\draw (-2,2)--(1,4);
\draw (-3,4)--(-2,2);
\draw (-3,4)--(0,2);
\draw (0,2)--(1,4);
\draw (-1,4)--(-2,2);
    \node at (-1,4.3) {$2j-1$};
    \node at (-2,1.7) {$2\ell $};
      \node at (0,1.7) {$2m$};
        \node at (1,4.3) {$2k-1$};
          \node at (-3,4.3) {$2i-1$};
  \end{tikzpicture}
    \caption{}
    \label{fig:case4}
\end{minipage}
\end{figure} 
\end{minipage}
\end{minipage}

So, in all cases, the induced subgraph on $\{2i-1,2j-1,2k-1,2\ell, 2m\}$ is connected. Since every vertex $u \in Y \setminus T$ is adjacent to at least $|X \setminus T|-1$ vertices in $X\setminus T$, i.e., $u$ is adjacent to at least two vertices among $\{2i-1,2j-1,2k-1\}$ and every vertex $v \in X\setminus T $ is adjacent to at least $|Y\setminus T|-1$ vertices in $Y\setminus T$, i.e., $v$ is adjacent to  $2\ell$ or $2m,$  $G[\bar T]$ is a connected graph.
\end{proof}

Next, we use the above lemma to compute the vertex connectivity of crown graphs. 

\begin{proposition}\label{prop:graph-connectivity}
    Let $G=C_{n,n}$ with $n \ge 3$. Then, $\kappa(G)=n-1.$ 
\end{proposition}

\begin{proof}
Let $T \subseteq [2n]$. Suppose that $\vert T\vert \le n-2$, then  $\vert X \setminus T\vert + \vert Y \setminus T \vert\ge n+2\ge 5$. Moreover, $|X \setminus T | \ge 2$ and $|Y \setminus T| \ge 2$. Thus, by Lemma \ref{lem:tech-connected}, $G[\bar T]$ is a connected graph. This implies that if we remove fewer than $n-1$ vertices from $G$, the remaining graph is still connected. \begin{minipage}{\linewidth}
  \begin{minipage}{0.5\linewidth}On the other hand, if we remove all the vertices of $X=\{1,3, \ldots,2n-1\}$ except one, say $2i-1$ for some $i \in [n]$, then $2i$ is not adjacent to any vertex in the remaining graph, i.e., the remaining graph is disconnected, see Figure \ref{fig:kappa}. Thus, the removal of a set of $n-1$ vertices makes the graph disconnected. 
\end{minipage}\hspace{0.6in}
\begin{minipage}{0.35\linewidth} \begin{figure}[H] \centering
    \begin{minipage}{\linewidth}
     \begin{tikzpicture}[scale=0.85]
\captionsetup{justification=centering,margin=2cm}        \filldraw (-1,4) circle(.3ex); 
        \filldraw (-1,2) circle(.3ex);
        \filldraw (-2,2) circle(.3ex);
        \filldraw (-4,2) circle(.3ex);
        \filldraw (0,2) circle(.3ex);
        \filldraw (2,2) circle(.3ex);
        \draw (-1,4)--(-4,2);
        \draw (-1,4)--(0,2);
        \draw (-1,4)--(-2,2);
        \draw (-1,4)--(2,2);
        \node at (-4,1.7) {$2$};
        \node at (-1,4.3) {$2i-1$};
        \node at (-2,1.7) {$2i-2$};
        \node at (-1,1.7) {$2i$};
        \node at (1,2) {$\ldots$};
        \node at (-3,2) {$\ldots$};
        \node at (0,1.7) {$2i+2$};
        \node at (2,1.7) {$2n$};
    \end{tikzpicture}
    \caption{}
    \label{fig:kappa}
\end{minipage}
\end{figure} 
\end{minipage}
\end{minipage} Hence, $\kappa(G) = n-1$. 
\end{proof}

Now, we give the complete description of sets of crown graphs with the cut-point property. 

\begin{theorem}\label{thm:cut-point-prop}
Let $G=C_{n,n}$ with $n \ge 3$. Let $T \subseteq [2n]$ be a nonempty set. 
     \begin{enumerate}
    \item[$(1)$] If $T \subseteq X$, then $T \in \mathcal{C}(G)$ if and only if either $T=X$ or $T=X\setminus \{2i-1\}$ for some $ i \in [n].$
    \item[$(2)$] If $T \subseteq Y$, then $T \in \mathcal{C}(G)$ if and only if either $T=Y$ or $T=Y\setminus \{2i\}$ for some $i \in [n].$
    \item[$(3)$]  If $T \not\subseteq X$ and $T \not\subseteq Y$, then $T  \in \mathcal{C}(G)$ if and only if  $T=A \cup \{ i+1~:~ i\in A\}$ for some $A \subseteq X$ with $|A|=n-2$. 
\end{enumerate}
    \end{theorem}
    \begin{proof} Let $T \subseteq [2n]$ be a nonempty set. \par (1) Assume that $T \subseteq X.$ If $|T| \le n-2$, then as we observe in the proof of Proposition \ref{prop:graph-connectivity}, $G[\bar T]$ is a connected graph. Thus, if $T \in \mathcal{C}(G)$, then $|T| \ge n-1,$ i.e., either $T=X$ or $T=X \setminus \{i\}$ for some $i \in X.$ 
        
        Conversely, suppose that $T=X$. For every $ i \in [n]$, $G[ \bar {T} \cup \{2i-1\}]$ is the graph shown in Figure \ref{fig:kappa}. It is clear from Figure \ref{fig:kappa} that $2i-1$ is a cut vertex in $G[ \bar {T} \cup \{2i-1\}]$. Thus, $T=X$ has the cut point property, hence $T \in \mathcal{C}(G).$ 
        
 \begin{minipage}{\linewidth}
  \begin{minipage}{0.45\linewidth} Next, suppose $T=X\setminus \{2i-1\}$ for some $i \in [n].$ Then, for every $ 2j -1 \in T$, $G[ \bar {T} \cup \{2j-1\}]$ is similar to the graph shown in Figure \ref{fig:cutset}. Notice that $2j-1$ disconnects $2i-1$ and $2i$, see Figure \ref{fig:cutset}. Therefore, $2j-1$ is a cut vertex in $G[ \bar {T} \cup \{2j-1\}]$. Hence, $T=X\setminus \{2i-1\}$ has the cut point property, i.e., $T \in \mathcal{C}(G).$ \end{minipage}\hspace{0.2in}
\begin{minipage}{0.35\linewidth} \begin{figure}[H] \centering
    \begin{minipage}{\linewidth}
     \begin{tikzpicture}[scale=0.9]
\captionsetup{justification=centering,margin=2cm}
        \filldraw (-1,4) circle(.3ex); 
    %\filldraw (-6,4) circle(.3ex);
    %\filldraw (-2,4) circle(.3ex); 
    %\filldraw (0,4) circle(.3ex); 
    %\filldraw (2,4) circle(.3ex); 
        \filldraw (-1,2) circle(.3ex);
        \filldraw (-2,2) circle(.3ex);
        \filldraw (-4,2) circle(.3ex);
        \filldraw (0,2) circle(.3ex);
        \filldraw (2,2) circle(.3ex);
\filldraw (-6,2) circle(.3ex);
\filldraw (-4,4) circle(.3ex);

        \draw (-1,4)--(-4,2);
        \draw (-1,4)--(0,2);
        \draw (-1,4)--(-2,2);
        \draw (-1,4)--(2,2);
        \draw (-1,4)--(-6,2);
        \draw[thick, loosely dashed,line width=1.5pt] (-4,4)--(-6,2);
        \draw[thick, loosely dashed,line width=1.5pt] (-4,4)--(-2,2);
        \draw[thick, loosely dashed,line width=1.5pt] (-4,4)--(-1,2);
        \draw[thick, loosely dashed,line width=1.5pt] (-4,4)--(2,2);
        \draw[thick, loosely dashed,line width=1.5pt] (-4,4)--(0,2);
     %\draw[dashed] (-0.5,3.75) -- (.5,3.75);
        \draw[dashed] (-0.25,3.75) -- (2,3.75);
        \draw[dashed] (-0.25,3.75) arc [start angle=0, end angle=180, x radius=0.75cm, y radius=1cm];
       % \draw[dashed] (2.5,3.75) arc [start angle=180, end angle=0, x radius=1cm, y radius=1cm];
        \draw[dashed] (-1.75,3.75) -- (-6,3.75);
        \draw[dashed] (2,5) -- (-6,5);

\node at (-4,1.7) {$2j$};
\node at (-2.5, 4.5) {$T$};
\node at (-4,4.3) {$2j-1$};
        \node at (-6,1.7) {$2$};
        %\node at (-6,4.3) {$1$};
        %\node at (-2.5,4.3) {$2i-3$};
        \node at (-1,4.3) {$2i-1$};
        %\node at (0.5,4.3) {$2i+1$};
        %\node at (1,4) {$\ldots$};
        %\node at (-5,4) {$\ldots$};
        %\node at (-3,4) {$\ldots$};
        %\node at (2,4.3) {$2n-1$};
        \node at (-2,1.7) {$2i-2$};
        \node at (-1,1.7) {$2i$};
        \node at (1,2) {$\ldots$};
        \node at (-5,2) {$\ldots$};
        \node at (-3,2) {$\ldots$};
        \node at (0,1.7) {$2i+2$};
        \node at (2,1.7) {$2n$};
    \end{tikzpicture}
    \caption{}
    \label{fig:cutset}
\end{minipage}
\end{figure} 
\end{minipage}
\end{minipage}

        \par (2) The proof follows in the same lines as the proof of part $(1)$ just by replacing the role of $X$ with the role of $Y$.
        \par (3)  First, assume that $T \not\subseteq X$, $T \not\subseteq Y$ and $T \in \mathcal{C}(G)$. We claim that $\vert X \setminus T\vert = 2 = \vert Y \setminus T\vert$. Suppose $|X\setminus T| \ge 2$, $|Y \setminus T| \ge 2$ with $|X\setminus T|+|Y \setminus T|\ge 5$, then $G[\bar T]$ is a connected graph, i.e., $T \not\in \mathcal{C}(G)$ which is a contradiction. Furthermore, if $\vert X\setminus T\vert = 1$,  i.e., there 
        \begin{minipage}{\linewidth}
  \begin{minipage}{0.5\linewidth} exists a unique $i \in [n]$ such that $2i -1\not\in T$, then for any $2j \in T$, the degree of $2j$ in $G[\bar T \cup \{2j\}]$ is at most one, see Figure~\ref{fig:Y(T)=1}, i.e., $2j$ is not  a cut vertex of $G[\bar T \cup \{2j\}]$, which is a contradiction to the fact that $T \in \mathcal{C}(G).$ Thus, $|X\setminus T|  \neq 1$. Similarly, $|Y \setminus T|\neq 1.$ This all together implies that  $|X \setminus T |=2=|Y \setminus T|$. \end{minipage}\hspace{0.15in}
\begin{minipage}{0.35\linewidth} \begin{figure}[H] \centering
    \begin{minipage}{\linewidth}
     \begin{tikzpicture}[scale=0.85]
\captionsetup{justification=centering,margin=2cm}
  \filldraw (-1,4) circle(.3ex); 
    \filldraw (4,2) circle(.3ex); 
        \filldraw (-1,2) circle(.3ex);
        \filldraw (-2,2) circle(.3ex);
        \filldraw (-4,2) circle(.3ex);
        \filldraw (0,2) circle(.3ex);
        \filldraw (2,2) circle(.3ex);
        \draw (-1,4)--(-4,2);
        \draw (-1,4)--(0,2);
        \draw (-1,4)--(-2,2);
        \draw [thick, loosely dashed,line width=1.5pt](-1,4)--(2,2);
        \draw (4,2)--(-1,4);
         \draw[dashed] (-.2,3.9) arc [start angle=0, end angle=180, x radius=.8cm, y radius=1.0cm];
         \draw[dashed](-1.8,3.9)--(-4.0,3.9);
         \draw[dashed](4.0,3.9)--(-.2,3.9);
        \node at (4,1.7) {$2n$};
        \node at (-4,1.7) {$2$};
        \node at (-1,4.3) {$2i-1$};
       \node at (1,4.5) {$T$};
        \node at (-2.25,1.7) {$2i-2$};
        \node at (-1,1.7) {$2i$};
        \node at (1,2) {$\ldots$};
        \node at (3,2) {$\ldots$};
        \node at (-3,2) {$\ldots$};
        \node at (.25 ,1.7) {$2i+2$};
        \node at (2,1.7) {$2j$};
    \end{tikzpicture}
    \caption{}
        \label{fig:Y(T)=1}
\end{minipage}
\end{figure} 
\end{minipage}
\end{minipage} 
\begin{minipage}{\linewidth}
\begin{minipage}{0.45\linewidth} \begin{figure}[H] \centering
    \begin{minipage}{\linewidth}
     \begin{tikzpicture}[scale=0.9]
\captionsetup{justification=centering,margin=2cm}
\filldraw (0,4) circle(.3ex); 
         \filldraw (1,2) circle(.3ex);
         \filldraw (-3.4,2) circle(.3ex);
         \filldraw (-2.4,2) circle(.3ex);
         \filldraw (-3.4,4) circle(.3ex);
          \filldraw (-2.4,4) circle(.3ex); 
           \filldraw (-1.4,4) circle(.3ex); 
            \filldraw (2.2,4) circle(.3ex);
             \filldraw (3.4,4) circle(.3ex); 
         \filldraw (3.4,2) circle(.3ex);
          \filldraw (1,4) circle(.3ex); 
\filldraw (0,2) circle(.3ex);
\filldraw (-1.4,2) circle(.3ex);
   \filldraw (2.2,2) circle(.3ex);
\draw (0,4)--(-1.4,2);
\draw (0,4)--(2.2,2);
\draw (1,4)--(2.2,2);
\draw (1,4)--(-1.4,2);
\draw [thick, loosely dashed,line width=1.5pt] (1,4)--(0,2);
\node at (-.2,4.3) {$2j-1$};
   \node at (-1.4,1.7) {$2i$};
   \node at (1.1,4.3) {$2\ell-1$};
   \node at (0,1.7) {$2j$};
   \node at(-.8, 4.8) {$T$};
   \node at(.5, 1.4) {$T$};
          \node at (2.2,1.7) {$2k$};
        \draw[dashed] (1.8,3.75) -- (3.6,3.75);
        \draw[dashed] (-0.7,2.25) -- (1.7,2.25);
        \draw[dashed] (2.8,2.25) -- (3.6,2.25);
        \draw[dashed] (1.8,3.75) arc [start angle=0, end angle=180, x radius=1.4cm, y radius=1.25cm];
        \draw[dashed] (-0.7,2.25) arc [start angle=0, end angle=-180, x radius=.5cm, y radius=1cm];
        \draw[dashed] (1.7,2.25) arc [start angle=-180, end angle=0, x radius=.5cm, y radius=1cm];
        \draw[dashed] (-1.7,2.25) -- (-3.6,2.25);
        \draw[dashed] (-1.0,3.75) -- (-3.6,3.75);
        \draw[dashed] (3.6,5.2) -- (-3.6,5.2);
        \draw[dashed] (3.6,1) -- (-3.6,1);
  \end{tikzpicture}
    \caption{}
    \label{fig:bat}
    \end{minipage}
\end{figure} 
\end{minipage} \hspace{0.15in}
\begin{minipage}{0.5\linewidth} 
    Next, suppose there exist $1 \le i<j \le n$ so that $2i-1, 2j \in T$ and $2i, 2j-1 \not\in T$. Then, the degree of $2j$ in $G[\bar T \cup \{2j\}]$ is one, see Figure~\ref{fig:bat}. Therefore, $2j$ is not a cut vertex in $G[\bar T \cup \{2j\}],$ which is a contradiction. Thus, $T=A \cup \{a+1~:~ a \in A\}$ for some $A \subseteq X$ with $|A|=n-2.$ 
   \end{minipage}
\end{minipage}      
\begin{minipage}{\linewidth}
  \begin{minipage}{0.5\linewidth}
  Conversely, if $T=A \cup \{a+1~:~ a \in A\}$ for some $A \subseteq X$ with $|A|=n-2.$ Let $u\in T$ be any vertex. Then, $G[\bar T \cup \{u\}]$ is a path graph on five vertices, and $u$ is the middle vertex of that path, see Figure \ref{A and A+1}. Therefore, $u$ is a cut vertex of $G[\bar T \cup \{u\}]$. Thus, $T$ has the cut point property; hence, $T \in \mathcal{C}(G).$\end{minipage}\hspace{0.15in}
\begin{minipage}{0.35\linewidth} \begin{figure}[H] \centering
    \begin{minipage}{\linewidth}
     \begin{tikzpicture}[scale=0.85]
\captionsetup{justification=centering,margin=2cm}
        \filldraw (-1.5,4) circle(.3ex); 
        \filldraw (0,4) circle(.3ex); 
        \filldraw (1.5,4) circle(.3ex); 
        \filldraw (-3,4) circle(.3ex); 
        \filldraw (3,4) circle(.3ex); 
        \filldraw (0,2) circle(.3ex);
        \filldraw (-1.5,2) circle(.3ex);
        \filldraw (-3,2) circle(.3ex);
        \filldraw (1.5,2) circle(.3ex);
        \filldraw (3,2) circle(.3ex);
        
        \draw[dashed] (-0.5,3.75) -- (.5,3.75);
        \draw[dashed] (2.5,3.75) -- (4,3.75);
        \draw[dashed] (-0.5,2.25) -- (.5,2.25);
        \draw[dashed] (2.5,2.25) -- (4,2.25);
        \draw[dashed] (-.5,3.75) arc [start angle=0, end angle=180, x radius=1cm, y radius=1cm];
        \draw[dashed] (.5,3.75) arc [start angle=180, end angle=0, x radius=1cm, y radius=1cm];
        \draw[dashed] (-.5,2.25) arc [start angle=0, end angle=-180, x radius=1cm, y radius=1cm];
        \draw[dashed] (.5,2.25) arc [start angle=-180, end angle=0, x radius=1cm, y radius=1cm];
        \draw[dashed] (-2.5,2.25) -- (-4,2.25);
        \draw[dashed] (-2.5,3.75) -- (-4,3.75);
        \draw[dashed] (4,5) -- (-4,5);
        \draw[dashed] (4,1) -- (-4,1);

        \draw (1.5,4) -- (-1.5,2);
        \draw [thick, loosely dashed, line width=1.5pt] (0,4) -- (-1.5,2);
       
        \draw [thick, loosely dashed,line width=1.5pt] (0,4) -- (1.5,2);
        \draw (-1.5,4) -- (1.5,2);

        \node at (0,4.3) {$u$};
        \node at (-1.5,4.3) {$2i-1$};
        \node at (-1.5,1.7) {$2i$};
        \node at (1.5,4.3) {$2j-1$};
        \node at (1.5,1.7) {$2j$};
        \node at (0,1.5) {A+1};
        \node at (0,4.7) {A};
    \end{tikzpicture}
    \caption{}
    \label{A and A+1}
\end{minipage}
\end{figure} 
\end{minipage}
\end{minipage} 
        Hence, the assertion follows.
    \end{proof}

Now, using the above description of the sets with the cut point property, we compute the Krull dimension of the quotient ring of the binomial edge ideals of the crown graphs.
\begin{theorem}\label{thm:dim} 
Let $G=C_{n,n}$ with $n \ge 3$. Then,  $\dim\left(\frac{S}{J_{C_{n,n}}}\right) = 2n + 1$. Moreover, \begin{enumerate}
    \item[$(1)$] if $T=\emptyset$, then $\dim\left(\frac{S}{P_{T}(G)}\right) = 2n + 1$. 
    \item[$(2)$] if $T=X$ or $T=Y$, then $\dim\left(\frac{S}{P_{T}(G)}\right) = 2n$. 
    \item[$(3)$] if $T=X\setminus \{2i-1\}$ or $T=Y\setminus \{2j\}$ for some $ i, j \in [n]$, then $\dim\left(\frac{S}{P_{T}(G)}\right) = n+3$. 
    \item[$(4)$] if $T=A\cup \{a+1~:~a \in A\}$ with $A \subseteq X$ and $|A|=n-2$, then $\dim\left(\frac{S}{P_{T}(G)}\right) = 6$.
\end{enumerate}
\end{theorem}

\begin{proof} It follows from \cite[Lemma~3.1]{HHHKR10} that for any $T \subseteq [2n]$, $$\dim\left(\frac{S}{P_{T}(G)}\right) = 2n-|T|+c_G(T).$$
\par (1)  Since $T=\emptyset$,  $c_G(T)=1.$ Thus, $\dim\left(\frac{S}{P_{T}(G)}\right) = 2n-0+1=2n+1.$
 \par (2) Assume $T=X$ (the case for $T=Y$ is similar). Then, $G[\bar T]$ is a graph on $n$ vertices with no edges. So, $c_G(T)=n,$ and hence, $\dim\left(\frac{S}{P_{T}(G)}\right) = 2n-n+n=2n.$
 \par (3) Assume $T=X \setminus \{2i-1\}$ for some $i \in [n]$ (the case for $T=Y\setminus \{2j\}$ for some $j \in [n]$ is similar). Then, $G[\bar T]$ has two connected components; see Figure \ref{fig:cutset}. Therefore, $\dim\left(\frac{S}{P_{T}(G)}\right) = 2n-(n-1)+2=n+3.$
\par (4) Assume $T=A\cup \{a+1~:~ a\in A\}$, where $A\subseteq X$ and $|A|=n-2$. Then, $G[\bar{T}]$ has two connected components (see Figure \ref{A and A+1}). Thus, $\dim\left(\frac{S}{P_{T}(G)}\right) = 2n-2(n-2)+2=6$.

By Theorem \ref{thm:cut-point-prop}, we have a complete description of elements of $\mathcal{C}(G)$. Since $\dim\left(\frac{S}{J_G}\right) =\max\left\{\dim\left(\frac{S}{P_{T}(G)}\right)~:~ T \in \mathcal{C}(G)\right\}$, we get that $ \dim\left(\frac{S}{J_G}\right) =2n+1,$ as desired. 
\end{proof}

 Using Theorem~\ref{thm:dim}, we immediately determine the big height of the binomial edge ideals of crown graphs. The \textit{big height} of an ideal is the largest height of the minimal primes of that ideal.
 
 \begin{corollary}\label{cor:height}
     Let $G=C_{n,n}$ with $n \ge 3$. Then, $\height(J_G)=2n-1$ and $\bigheight(J_G)=4n-6.$
 \end{corollary}
 \begin{proof}
 It is well known that for a homogenous  ideal $I$ in $S$, $\height(I)=\dim(S)-\dim\left(\frac{S}{I}\right).$ We know that $\dim(S)=4n$. Therefore,  $$\height(J_G)=\dim(S)-\dim\left(\frac{S}{J_G}\right)=4n-(2n+1)=2n-1$$ and \begin{align*}
     \bigheight(J_G)&=\max\left\{ \height(P_T(G))~:~ T \in \mathcal{C}(G)\right\}\\&=\max\left\{\dim(S)-\dim\left(\frac{S}{P_T(G)}\right)~:~T \in \mathcal{C}(G)\right\}\\& =\dim(S)-\min\left\{\dim\left(\frac{S}{P_T(G)}\right)~:~T \in \mathcal{C}(G)\right\}\\&=4n-6.
 \end{align*} Hence, the assertion follows. 
 \end{proof}

\section{Projective dimension and big height of binomial edge ideals}\label{depth} The projective dimension of the quotient $\frac{S}{I}$ of a homogeneous ideal $I \subset S$ is known to be greater than or equal to the big height of $I$. When $\frac{S}{I}$ is Cohen-Macaulay, then the projective dimension of $\frac{S}{I}$ is the same as the $\bigheight(I)$. We want to understand which class of ideals(other than Cohen-Macaulay) has the same projective dimension as the big height of $I$. Specifically, we are looking at this question in the context of binomial edge ideals. In this section, we will explore this question for binomial edge ideals associated with various classes of graphs. 

The first known result about the projective dimension of binomial edge ideals is attributed to Ene, Herzog, and Hibi \cite{EHH11}. They investigated the Cohen-Macaulayness of binomial edge ideals of block graphs. As a first step in this direction, we verify that the quotient of binomial edge ideals of block graphs has the same projective dimension as the big height.

\begin{proposition}
    Let $G$ be a connected block graph on the vertex set $[n]$. Then, $$\pd\left(\frac{S}{J_G}\right)=\bigheight(J_G)=n-1.$$
\end{proposition}
\begin{proof}
   Based on \cite[Theorem 1.1]{EHH11}, we have $\pd\left(\frac{S}{J_G}\right)=n-1$. This implies that $\bigheight(J_G) \le \pd\left(\frac{S}{J_G}\right)=n-1$. 
To demonstrate that $\bigheight(J_G) \ge n-1$, we first note that since $G$ is a connected graph, it follows from \cite[Lemma 3.1]{HHHKR10} that $\height(P_\emptyset(G))=n-1$. Therefore, $n-1 = \height(P_\emptyset(G)) \le  \bigheight(J_G)$. 
Thus, the assertion follows.
\end{proof}

The projective dimension of the binomial edge ideals of cycles is well-documented by Zafar and Zahid \cite{betti}. In the following, we will confirm that the quotient of the binomial edge ideal of cycles has the same projective dimension as the big height. 

\begin{proposition} \label{prop:pd-bh-cycle}
    Let $G=C_n$ with $n \ge 4$. Then $\pd\left( \frac{S}{J_G}\right)=\bigheight (J_G)=n$.
\end{proposition}
\begin{proof}
    Based on \cite[Corollary 16]{betti}, we know that $\pd\left(\frac{S}{J_G}\right)=n$. This means that $\bigheight(J_G) \le \pd\left(\frac{S}{J_G}\right)=n$. To show that $\bigheight(J_G) \ge n$, we first observe that the set $T=\{1,3\}$ has the cut point property. Additionally, $G[\bar T]$ forms a disconnected graph with two connected components, meaning that $c_G(T)=2$. %\textcolor{violet}{(see Figure~\ref{fig: pd(C_n) = bight(C_n)})}. 
    As a result of \cite[Lemma 3.1]{HHHKR10}, we deduce that $\height(P_T(G))=n+|T|-c_G(T)=n$. Therefore, we have $n = \height(P_T(G)) \le  \bigheight(J_G)$. Hence, the claim follows.
\end{proof}

Kumar and Sarkar \cite{KS20} determined the projective dimension of the binomial edge ideals for both wheel graphs and complete multipartite graphs. They also provided a formula for the projective dimension of the binomial edge ideals for the join product of graphs. In the next few results, we examine the cone of a graph and complete multipartite graphs, focusing on when the quotient of the binomial edge ideal has the same projective dimension as the big height. The {\it cone} of a graph $G$ is a graph on the vertex set $V(G) \sqcup \{v\}$ and edge set $E(G) \sqcup \{ \{v,u\}~:~ u \in V(G)\}.$
\begin{proposition}\label{prop:pd-bh-cone}
    Let $G$ be a connected graph on the vertex set $[n]$ and let $v$ be a new vertex. Let $H=v*G$ and $S_H=S[x_v,y_v]$. Then, we have the following: 
    \begin{enumerate}
        \item[$(1)$] $\pd\left( \frac{S_H}{J_H}
        \right) =2+\pd\left( \frac{S}{J_G}
        \right).$
        \item[$(2)$] if $\height(P_T(G)) <n-1$ for all $T \in \mathcal{C}(G) \setminus \{\emptyset\},$ then $\bigheight(J_H)=n.$
        \item[$(3)$] if $\height(P_T(G)) \ge n-1$ for some  $T \in \mathcal{C}(G) \setminus \{\emptyset\},$ then $\bigheight(J_H)=2+\bigheight(J_G).$
        \item[$(4)$] $\pd\left( \frac{S_H}{J_H}\right)=\bigheight(J_H)$ if and only if  $\pd\left( \frac{S}{J_G}\right)=\bigheight(J_G)$ and    $\height(P_T(G)) \ge n-1$ for some  $T \in \mathcal{C}(G) \setminus \{\emptyset\}$.  \end{enumerate}
\end{proposition}
\begin{proof}
    (1) This part follows using the Auslander-Buchsbaum formula in  \cite[Theorem 3.4]{KS20}. 
    \par (2-3) Based on \cite[Proposition 4.1]{DS15}, we have $$\mathcal{C}(H)=\{\emptyset\} \cup \left\{\{v\} \cup T~:~ T \in \mathcal{C}(G) \setminus \{\emptyset\}\right\}.$$ Using \cite[Lemma 3.1]{HHHKR10}, we know that for any $A \in \mathcal{C}(H)$, $\height(P_A(H))=(n+1)+|A|-c_H(A)$. Consequently, $\height(P_\emptyset(H))=n$, and for any $A \in \mathcal{C}(H)$ with $A \neq \emptyset$, we have $$\height(P_A(H))=2+n+|A\setminus \{v\}| -c_G(A \setminus \{v\})=2+\height(P_{A \setminus \{v\}}(G)).$$
In this way, we get the respective conclusion by using the hypothesis of the respective part of the definition of the big height. 
\par (4) Using parts $(1)$ and $(3)$, we get $\pd\left( \frac{S_H}{J_H}\right)=\bigheight(J_H)$ if $\pd\left( \frac{S}{J_G}\right)=\bigheight(J_G)$ and $\height(P_T(G)) \geq n-1$ for some $T \in \mathcal{C}(G) \setminus \{\emptyset\}$.

 Conversely, suppose $\pd\left( \frac{S_H}{J_H}\right)=\bigheight(J_H)$. Since $G$ is a connected graph, $\pd\left(\frac{S}{J_G} \right) \geq \bigheight(J_G) \geq \height(P_\emptyset(G)) = n-1$. Using part $(1)$, we get $\pd\left(\frac{S_H}{J_H} \right) \geq n+1$, which means $\bigheight(J_H) \geq n+1$. According to parts $(2-3)$, this implies that $\height(P_T(G)) \geq n-1$ for some $T \in \mathcal{C}(G) \setminus \{\emptyset\}$, and hence, $\bigheight(J_H)=2+\bigheight(J_G)$. We get the desired result using part $(1)$. 
\end{proof}
As an immediate consequence of Proposition \ref{prop:pd-bh-cone}, we confirm that the projective dimension of the quotient of the binomial edge ideal of wheel graphs is the same as the big height. Recall that the {\it wheel graph} on $n+1$ vertices is the cone of a cycle graph on $n$ vertices. 

\begin{corollary}
    Let $H=v*C_n$ be the wheel graph on the vertex set $[n]\cup \{v\}$ with $n \ge 4.$ Then, $\pd\left( \frac{S_H}{J_H}\right)=\bigheight(J_H)=n+2,$ where $S_H=S[x_v,y_v].$
\end{corollary}
\begin{proof}
   By Proposition \ref{prop:pd-bh-cycle}, we know that  $\pd\left( \frac{S}{J_{C_n}}\right)=\bigheight(J_{C_n})=n.$ This implies that there exists $T \in \mathcal{C}(C_n)$ so that $\height(P_T(C_n)) =n\ge n-1.$ Thus, using Proposition \ref{prop:pd-bh-cone}, $\pd\left( \frac{S_H}{J_H}\right)=\bigheight(J_H)=n+2.$
\end{proof}

\begin{proposition}
    Let $G$ be a disconnected graph on the vertex set $[n]$ and let $v$ be a new vertex. Let $H=v*G$ and $S_H=S[x_v,y_v]$. Then, we have the following: 
    \begin{enumerate}
        \item[$(1)$] $\pd\left( \frac{S_H}{J_H}
        \right) =\max\left\{n, 2+\pd\left( \frac{S}{J_G}
        \right)\right\}.$
           \item[$(2)$]  $\bigheight(J_H)=\max\left\{n,2+\bigheight(J_G)\right\}.$
        \item[$(3)$] if $\bigheight(J_G) \ge n-2$, then $\pd\left( \frac{S_H}{J_H}\right)=\bigheight(J_H)$ if and only if  $\pd\left( \frac{S}{J_G}\right)=\bigheight(J_G).$ %\item[$(4)$] if $\pd\left( \frac{S}{J_G}\right) \le n-2$, then $\pd\left( \frac{S_H}{J_H}\right)=\bigheight(J_H)$.
        \item[$(4)$] if $\bigheight(J_G) < n-2$, then $\pd\left( \frac{S_H}{J_H}\right)=\bigheight(J_H)$ if and only if  $\pd\left( \frac{S}{J_G}\right)\le n-2.$   \end{enumerate}
\end{proposition}
\begin{proof}
    (1) This part follows using the Auslander-Buchsbaum formula in  \cite[Theorem 3.9]{KS20}. 
    \par (2) Based on \cite[Proposition 4.5]{DS15}, we have $$\mathcal{C}(H)=\{\emptyset\} \cup \left\{\{v\} \cup T~:~ T \in \mathcal{C}(G)\right\}.$$ Using \cite[Lemma 3.1]{HHHKR10}, we know that for any $A \in \mathcal{C}(H)$, $\height(P_A(H))=(n+1)+|A|-c_H(A)$. Consequently, $\height(P_\emptyset(H))=n$, and for any $A \in \mathcal{C}(H)$ with $A \neq \emptyset$, we have $$\height(P_A(H))=2+n+|A\setminus \{v\}| -c_G(A \setminus \{v\})=2+\height(P_{A \setminus \{v\}}(G)).$$
Thus, \begin{align*}
    \bigheight(J_H)&= \max\left\{\height(P_\emptyset(H)), \max\left\{ \height(P_A(H))~:~ A\in \mathcal{C}(H)\right\} \right\}\\& =\max\left\{n,\max\left\{ 2+\height(P_{A \setminus \{v\}}(G))~:~ A \setminus \{v\} \in \mathcal{C}(G)\right\}\right\}\\& =\max\left\{n, 2+\max\left\{ \height(P_{T}(G))~:~ T \in \mathcal{C}(G)\right\}\right\} \\& =\max\left\{n, 2+\bigheight(J_G)\right\} .
\end{align*} 
\par (3) Assume that $\bigheight(J_G) \ge n-2.$ Then, $\pd\left(\frac{S}{J_G}\right) \ge \bigheight(J_G) \ge n-2.$ Therefore, by parts $(1-2)$, $\pd\left( \frac{S_H}{J_H}
        \right) =\max\left\{n, 2+\pd\left( \frac{S}{J_G}
        \right)\right\}=2+\pd\left( \frac{S}{J_G}
        \right) $
 and            $\bigheight(J_H)=\max\left\{n,2+\bigheight(J_G)\right\}=2+\bigheight(J_G).$ Thus, the desired result follows. 

\par (4) Assume that $\bigheight(J_G) < n-2$. Then, by part (2), $\bigheight(J_H) =n$. Therefore, by part (1), $\pd\left( \frac{S_H}{J_H} \right) =\bigheight(J_H)$ if and only if $\pd\left( \frac{S_H}{J_H} \right) =n$ if and only if  $\pd\left( \frac{S}{J_G}\right)\le n-2$. Thus, the desired result follows. 
\end{proof}

Next, we study complete multipartite graphs where the quotient of the binomial edge ideals has the same projective dimension as the big height.

\begin{proposition}\label{prop:com-bip}
    Let $G=K_{n_1,\ldots,n_k}$ be a complete multipartite graph with $2 \le n_1 \le \cdots \le n_k.$ Then, we have the following: \begin{enumerate}
        \item[$(1)$] $\pd\left(\frac{S}{J_G}\right)=2(n_2+\cdots+n_k)+n_1-2.$
        \item[$(2)$] $\bigheight(J_G)=2(n_2+\cdots + n_k).$
        \item[$(3)$] $\pd\left(\frac{S}{J_G}\right)=\bigheight(J_G)$ if and only if $n_1=2.$
    \end{enumerate}
\end{proposition}
\begin{proof}
    (1) We use the Auslander-Buchsbaum formula in  \cite[Corollary 4.5]{KS20} to get the projective dimension.
    \par (2) 
Let $V_1 \sqcup \cdots \sqcup V_k$ be a partition of the vertex set $V(G)$, where $|V_i|=n_i$ for $ i \in [k]$. According to \cite[Lemma 2.2]{Oh11}, we have $\mathcal{C}(G)=\{\emptyset\} \cup \{V(G) \setminus V_i~:~ i \in [k]\}$. It's important to observe that the connected components of $G[\bar {V(G) \setminus V_i}]$ correspond to the vertices that are part of $V_i$. Therefore, using \cite[Lemma 3.1]{HHHKR10}, we get  
$$\height(P_{V(G)\setminus V_i}(G))= (n_1+\cdots +n_k)+|V(G) \setminus V_i|-|V_i|=2(n_1+\cdots+n_k)-2n_i.$$ 
Applying the definition of the big height, we get 
\[
\begin{aligned}
\bigheight(J_G) & =\max\{ \height(P_T(G))~:~ T \in \mathcal{C}(G)\}\\
 & = \max\{ n_1+\cdots+n_k-1, \max\{2(n_1+\cdots+n_k)-2n_i~:~i \in [k]\}\}\\
 & =2(n_2+\cdots+n_k).
\end{aligned}
\]
    \par (3) This part clearly follows using parts $(1-2). $
\end{proof}

Malayeri et al. \cite[Theorem~5.3]{small-depth} proved that the projective dimension of the quotient of the binomial edge ideals is at most $2n-4$. They also \cite{small-depth} characterized graphs whose binomial edge ideals have the maximal projective dimension. Here, we verify that for binomial edge ideals of maximal projective dimension, the projective dimension is the same as the big height. 

\begin{proposition}
    Let $G$ be a graph on the vertex set $[n]$ with $\pd\left(\frac{S}{J_G}\right)=2n-4.$ Then, $$\pd\left(\frac{S}{J_G}\right)=\bigheight(J_G)=2n-4.$$
\end{proposition}
\begin{proof}
First, it is sufficient to prove that $\bigheight(J_G) \ge 2n-4$. From \cite[Theorem 5.3]{small-depth} and the Auslander-Buchsbaum formula, we know that $\pd\left(\frac{S}{J_G}\right)=2n-4$ if and only if $G=2K_1*H$, where $H$ is any graph on $n-2$ vertices. According to \cite[Propositions 4.5, 4.14]{DS15}, we also know that  $T=V(H) \in \mathcal{C}(G)$. It is important to note that $G[\bar T]$ is a disconnected graph with two isolated vertices. Therefore, we have $\height(P_T(G))=n+|T|-c_G(T)=n+(n-2)-2=2n-4$, and hence, $\bigheight(J_G) \ge \height(P_T(G)) =2n-4.$ This concludes the proof.
 \end{proof}

%%%%%%%%%

In this sequel, we establish that the projective dimension of the quotient of the binomial edge ideals of crown graphs is equal to the big height. We recall the definitions of the graph family $\mathcal{G}_T$, where $T\subseteq[n]$, and the $D_5$-type graph below from \cite{D5}, which are useful in proving Theorem~\ref{thm: lowerbounddepth}.

    Let $T\subseteq [n]$ with $|T| = n-2$. In \cite{D5}, the authors defined a family of graphs on $[n]$, denoted by $\mathcal{G}_T$. For each $G\in\mathcal{G}_T$, there exist two non-adjacent vertices $u$ and $w$ of $G$ outside $T$, and three disjoint subsets of $T$, denoted as $V_0, V_1$, and $V_2$ where $V_1, V_2\ne \emptyset$ and $\bigcup_{i=0}^2 V_i = T$. In addition, it satisfies $N_G(u) = V_0\cup V_1$, $N_G(w) = V_0\cup V_2$, and $\{v_1, v_2\} \in E(G)$ for every $v_1\in V_1$ and every $v_2\in V_2$.

The authors in \cite{D5} defined $D_5$-type graphs. A graph $G$ on $[n]$ is said to be a {\it $D_5$-type graph} if  $G\ne G'*2K_1$ for any graph $G'$ and  either $G\in\mathcal{G}_T$ for some $T\subseteq [n]$ or  $G = H*3K_1$, for some graph $H$ or $G = H*(K_1\dot\cup K_2)$, for some graph $H$.

%%%%%%% 
\begin{theorem}\label{thm: lowerbounddepth}
Let $G=C_{n,n}$ for $n \ge 3.$ Then, $$\pd\left( \frac{S}{J_G}\right)=\bigheight(J_G)=4n-6.$$
\end{theorem}
\begin{proof}
Based on \cite[Theorem~5.2]{small-depth} and the Auslander-Buchsbaum formula, we establish that $\pd\left(\frac{S}{J_G}\right)\le 4n-4$. If it were the case that  $\pd\left(\frac{S}{J_G}\right) = 4n-4$, then according to \cite[Theorem~5.3]{small-depth}, it would follow that $G = H* 2K_1$, where $H$ is a graph on $2n-2$ vertices. However, this would lead to the contradiction that $G$ must have two non-adjacent vertices of degree $2n-2$, whereas we know that $\deg_G(v) = n-1$ for all $v\in V(G)$. Therefore, we conclude that $\pd\left(\frac{S}{J_G}\right)\le 4n-5$. Again, assuming  $\pd\left(\frac{S}{J_G}\right) = 4n-5$, by \cite[Theorem~5.4]{D5}, we deduce that $G$ is a $D_5-$type graph. This means that either $G \in \mathcal{G}_T$ for some $T\subseteq [2n]$ with $|T|=2n-2$, or $G = H*3K_1$ for some graph $H$ on $2n-3$ vertices, or $G = H'*(K_1\overset{.}{\cup}K_2)$ for some graph $H'$ on $2n-3$ vertices. One could observe that $G$ cannot be the latter two graphs. In fact, if $G = H*3K_1$ or $G = H'*(K_1\overset{.}{\cup}K_2)$, then $G$ must have a vertex of degree $2n-3$, which contradicts the fact that $\deg_G(v) = n-1$ for all $v\in V(G).$

Suppose $G\in\mathcal{G}_T$ for some $T\subseteq [2n]$ with $|T|=2n-2$. According to the construction of $\mathcal{G}_T$, there are two non-adjacent vertices $u$ and $w$ in $[2n]\backslash T$, and three disjoint subsets of $T$ denoted as $V_0$, $V_1$, and $V_2$, where $V_1$ and $V_2$ are not empty and $\displaystyle\bigcup_{i=0}^2 V_i = T$. Additionally, $N_G(u) = V_0\cup V_1$, $N_G(w) = V_0\cup V_2$, and $\{v_1, v_2\} \in E(G)$ for $v_1\in V_1$ and $v_2\in V_2$. As $u$ and $w$ are non-adjacent, we will consider the following cases:
\begin{itemize}
    \item If $u = 2i-1$ and $w = 2i$ for some $i\in [n]$, then we have $V_0=\emptyset$, $V_1 = N_{G}(u)= \{2j~:~ j \in [n], j\ne i\}$, and $V_2 = N_{G}(w) = \{2k-1~:~ k\in [n], k\ne i\}$. It is clear that for any $k\in [n]$ with $k \neq i$, $2k\in V_1$, $2k-1\in V_2$, and $\{2k-1,2k\}\notin E(G)$, which leads to a contradiction. 
    \item If $u=2i-1$, $w=2j-1$ for some $i,j\in [n]$, then $T=N_{G}(u)\cup N_{G}(w) = \{2i~:~ i\in [n]\}$, which is a contradiction to the fact that $|T|=2n-2$.
    \item If $u=2i$, $w=2j$ for some $i,j\in [n]$, then $T=N_{G}(u)\cup N_{G}(w) = \{2i-1~:~ i\in [n]\}$, which is a contradiction to the fact that $|T|=2n-2$.
\end{itemize}

Thus, all possible cases lead us to the conclusion that $G$ is not a $D_5$-type graph, and hence $\pd\left(\frac{S}{J_G}\right)\le 4n-6$. By Corollary~\ref{cor:height}, $\bigheight(J_G)=4n-6$. This implies that $4n-6 =\bigheight(J_G) \le \pd\left(\frac{S}{J_G}\right)\le 4n-6$, and hence, the desired result follows.
\end{proof}

Proposition \ref{prop:com-bip} demonstrates that the projective dimension and the big height differ for the binomial edge ideal of $K_{3,3,3}$. This leads to a natural question: 
\begin{question}
    For which class of graphs does $\pd\left( \frac{S}{J_G}\right) = \bigheight(J_G)$ hold?
\end{question}

\section{v-number of binomial edge ideals}\label{localvnumber}
Let $I$ be a proper non-zero homogeneous ideal in $S$, and let $\mathfrak{p} \in \Ass(I)$. The \textit{local Vasconcelos number} of $I$ with respect to $\mathfrak{p}$, denoted by $\vn_{\mathfrak{p}}(I)$, is the least possible degree of a homogeneous element $f$ such that $I:f=\mathfrak{p}$, i.e., $$\vn_{\mathfrak{p}}(I):= \min \{ d~:~ \exists f \in S_d \text{ so that } I:f=\mathfrak{p}\},$$  where $S_d$ is the $\KK$-vector space spanned by all the monomials of degree $d$.  The \textit{Vasconcelos number}, abbreviated as the \textit{$\vn$-number}, of $I$, denoted by $\vn(I)$, is the minimum of the local Vasconcelos numbers of $I$, i.e., $$\vn(I):= \min \{ \vn_\mathfrak{p}(I)~:~ \mathfrak{p}\in \Ass(I)\}.$$
The concept of the $\vn-$number of homogeneous ideals was first introduced by Cooper et al. \cite{CSTPV20}. In this section, we examine the local $\vn-$number of the binomial edge ideals of graphs. The initial investigation of the local $\vn-$number of the binomial edge ideals was undertaken by Jaramillo-Velez and Seccia \cite{JS23}, where they explored the local $\vn-$number of the binomial edge ideals of connected graphs in relation to the binomial edge ideal of a complete graph with the same vertex set. Specifically, they proved that if $G$ is a connected graph, then $\vn_{P_\emptyset(G)}(J_G) = \gamma_c(G)$, where $\gamma_c(G)$ represents the connected domination number of $G$. Simultaneously, Ambhore et al. \cite{vnumberbinomial} independently derived similar and additional results.

\subsection{v-number of binomial edge ideals of crown graphs:}\label{localvcrown} In this subsection, we investigate the $\vn$-number of the binomial edge ideals of crown graphs.

The following lemma may be familiar to researchers in graph theory, but we couldn't find a reference. Therefore, we are providing a proof for completeness.
\begin{lemma}
    \label{lem:mindomset}
  Let $G=C_{n,n}$ with $n \ge 3.$  Then, $\gamma_c(G)=4.$ 
\end{lemma}
\begin{proof}
First, we prove that any set $U$ with $|U| \le 3$ is not a connected dominating set. If $|U|=1$, then $U=\{2i-1\}$ or $U=\{2j\}$ for some $1 \le i,j\le n$. Then, either $2i$ or $2j-1$ are not in the neighbor of vertices of $U$. Therefore, $U$ is not a dominating set. Assume  $2 \le|U|$. If $U \subseteq X$ or $U \subseteq Y$, then $G[U]$ is a disconnected graph, i.e., $U$ is not a connected dominating set. Assume $G[U]$  is connected. Then, either $U=\{2k-1,2\ell\}$ or $\{2k-1,2\ell,2m-1\}$ or $\{2j,2k-1,2\ell\}$ for some $1 \le j,k,\ell,m \le n$ with $j \ne k, k \ne \ell, \ell \neq m$. Then, $2\ell-1$ is not a neighbor of any vertices of $U$ in the first two cases, and $2k$ is not a neighbor of any vertices of $U$ in the last case. Thus, if $|U| \le 3,$ then $U$ is not a connected dominating set. 
 \begin{minipage}{\linewidth}
  \begin{minipage}{0.6\linewidth}
Next, take $U=\{1,2,4,5\}$.  We claim that $U$ is a minimally connected dominating set of $G$. Notice that $G[U]$ is a path graph on four vertices as the edges of $G[U]$ are $\{1,4\},\{4,5\},\{2,5\}$. Therefore, $G[U]$ is a connected graph. Recall from Lemma \ref{lem:tech-connected} that $X=\{1,3,\ldots,2n-1\}$ and $Y=\{2,4,\ldots,2n\}.$ Now, every vertex of $X$ is a neighbor $2$ or $4$, and every vertex of $Y$ is a neighbor of $1$ or $5$, see Figure \ref{fig: domC_5,5}. \end{minipage}\hspace{0.15in}
\begin{minipage}{0.35\linewidth} \begin{figure}[H] \centering
    \begin{minipage}{\linewidth}
     \begin{tikzpicture}
\captionsetup{justification=centering,margin=2cm}
        \filldraw (-1,4) circle(.3ex); 
         \filldraw (0,4) circle(.3ex); 
          \filldraw (2,4) circle(.3ex); 
           \filldraw (-2,4) circle(.3ex); 
            \filldraw (3,4) circle(.3ex); 
\filldraw (0,2) circle(.3ex);
\filldraw (-1,2) circle(.3ex);
\filldraw (-2,2) circle(.3ex);
\filldraw (2,2) circle(.3ex);
\filldraw  (3,2) circle(.3ex);
\draw [dashed](3,4)--(-1,2);
\draw [dashed](3,4)--(-2,2);
\draw [dashed](3,4)--(0,2);
\draw [dashed](3,4)--(2,2);
\draw  (-2,4)--(-1,2);
\draw  (0,4)--(-1,2);
\draw [dashed](2,4)--(-1,2);
\draw [dashed](-1,4)--(3,2);
\draw [dashed](-1,4)--(-2,2);
\draw [dashed](-1,4)--(0,2);
\draw [dashed](-1,4)--(2,2);
\draw [dashed](-2,4)--(3,2);
\draw [dashed](2,4)--(3,2);
\draw [dashed](0,4)--(3,2);
\draw [dashed](-2,4)--(0,2);
\draw [dashed](-2,4)--(2,2);
\draw [dashed](2,4)--(0,2);
\draw [dashed](2,4)--(-2,2);
\draw [dashed](0,4)--(2,2);
\draw (0,4)--(-2,2);
\node at (-2,4.3) {\textbf{1}};
\node at (-1,4.3) {$3$};
\node at (0,4.3) {\textbf{5}};
\node at (3.2,4.3) {$2n$-$1$};
   \node at (-2,1.7) {\textbf{2}};
   \node at (-1,1.7) {\textbf{4}};
   \node at (0,1.7) {$6$};
   \node at (1.9,1.7) {$2n$-$2$};
     \node at (1.9,4.3) {$2n$-$3$};
          \node at (3,1.7) {$2n$};
         \node at (.8,1.7){$\ldots$};
         \node at (.8,4.3){$\ldots$};
           \end{tikzpicture}    \caption{}
    \label{fig: domC_5,5}
\end{minipage}
\end{figure} 
\end{minipage}
\end{minipage} 
Thus, $U$ is a connected dominating set. Hence, $\gamma_c(G)=4$. 
\end{proof}

\begin{proposition}\label{prop:Crownv}
    Let $G = C_{n,n}$ with $n\ge 4$. Then,  $\vn_{P_T(G)} (J_G)\le 4$ for all $T \in\mathcal{C}(G)$. 
\end{proposition}
\begin{proof} 
We know by \cite[Theorem 3.2]{JS23} (see also \cite[Theorem~3.6]{vnumberbinomial}) that if $G$ is a connected graph, then $\vn_{P_\emptyset(G)} (J_G) = \gamma_c(G)$. Therefore, by Lemma~\ref{lem:mindomset} we see $\vn_{P_\emptyset(G)} (J_G) = 4$. Let $T \in \mathcal{C}(G)$ be a nonempty set. Then, by Theorem \ref{thm:cut-point-prop}, we have the following cases:  
    \begin{itemize}
        \item Suppose $T =X$. Take $f=(x_2y_4-x_4y_2)(x_{2i}y_{2i+2}-x_{2i+2}y_{2i})$ for some $3\le i \le n-1$. Notice that $x_2y_4-x_4y_2 \in P_{\emptyset}(G)$. Since $x_2,y_2 \in P_Y(G)$ and   for any $1 \le j \le n$, either $x_2,y_2  \in  P_{Y\setminus \{2j\}}(G)$ or $x_4,y_4  \in  P_{Y\setminus \{2j\}}(G)$, we get $x_2y_4-x_4y_2 \in P_{Y}(G), P_{Y\setminus \{2j\}}(G)$. Also, if $T'=A \cup \{a+1~:~a \in A\}$ for some $A \subseteq X$ with $|A|=n-2$, then either $x_2,y_2 \in P_{T'}(G)$ or $x_4,y_4 \in P_{T'}(G)$ or $x_{2i},y_{2i} \in P_{T'}(G)$. This implies that $f\in P_{T'}(G).$ Finally, if $T'=X \setminus \{2j-1\}$ for some $1 \le j \le n$, then  either $x_2y_4-x_4y_2 \in P_{T'}(G)$ or $x_{2i}y_{2i+2}-x_{2i+2}y_{2i} \in P_{T'}(G)$  which implies that $f \in P_{T'}(G).$  Thus, $f \in P_{T'}(G)$ for all $T' \in \mathcal{C}(G) \setminus \{X\}$, and $f \not\in P_X(G)$. So, $(J_G : f)= P_X(G)$, and hence, $\vn_{P_X(G)} (J_G) \le 4$.
        \item Suppose $T=Y$. Take $f=(x_1y_3-x_3y_1)(x_{2i-1}y_{2i+1}-x_{2i+1}y_{2i-1})$ for some  $3\le i \le n-1$. Using similar argument as above, we get $(J_G : f)= P_Y(G)$, and hence, $\vn_{P_Y(G)}(J_G) \le 4$
        
        \item Suppose $T=X \setminus \{2i-1\}$ for some $i \in [n].$ Take $f=x_{2i-1}x_{2k}( x_{2i}y_{2j}-x_{2j}y_{2i})$, where $i,j,k$ are distinct with $1\le k,j \le n$. Notice that $x_{2i-1} \in P_X(G), P_{X \setminus \{ 2 \ell-1\}}(G)$ for any $\ell \ne i.$ This implies that $f \in P_X(G)$ and $f \in P_{X \setminus \{ 2 \ell-1\}}(G)$ for any $\ell \ne i.$ Since  $x_{2i}y_{2j}-x_{2j}y_{2i} \in P_\emptyset(G)$, we get $f \in P_{\emptyset}(G)$. Next, $x_{2k} \in P_Y(G)$ and either $x_{2k}$ or $x_{2i},y_{2i}$ are elements of $P_{Y\setminus \{2\ell\}}(G)$, we know that $f \in P_Y(G), P_{Y\setminus \{2\ell\}}(G)$ for any $1 \le \ell \le n.$ Finally, if $T'=A \cup \{a+1~:~a \in A\}$ for some $A \subseteq X$ with $|A|=n-2$, then either $x_{2i},y_{2i} \in P_{T'}(G)$ or $x_{2j},y_{2j} \in P_{T'}(G)$ or $x_{2k} \in P_{T'}(G)$. In either case, we get $f\in P_{T'}(G).$ Thus, $f \in P_{T'}(G)$ for all $T' \in \mathcal{C}(G) \setminus \{X\}$, and $f \not\in P_X(G)$. So, $(J_G : f)= P_X(G)$, and hence, $\vn_{P_{X \setminus \{2i-1\}}(G)} (J_G) \le 4$. 
        
        \item Suppose $T=Y \setminus \{2i\}$ for some $i \in [n]$. Take $f=x_{2k-1}x_{2i}
        (x_{2i-1}y_{2j-1}-x_{2j-1}y_{2i-1})$ where $i,j,k$ are distinct with $1\le j,k \le n$. Using the similar arguments to the previous case, we get $(J_G: f)= P_{Y\backslash\{2i\}}(G)$. Thus, $\vn_{P_{Y\setminus \{2i\}}(G)}(J_G) \le 4 $.
        
        \item Suppose $T=A \cup \{a+1~:~ a\in A\}$ where $A \subseteq X$ with $|A|=n-2$ and $2i-1,2i,2j-1,2j \notin T$. Take  $f=(x_{2i-1}y_{2j-1}-x_{2j-1}y_{2i-1})(x_{2i}y_{2j}-x_{2j}y_{2i})$. Then $x_{2i-1}y_{2j-1}-x_{2j-1}y_{2i-1}$ is in $P_\emptyset(G), P_X(G), P_{X\setminus \{2\ell -1\}}$ for every $1 \le \ell \le n$. Also,  $x_{2i}y_{2j}-x_{2j}y_{2i}\in P_Y(G)$ and $P_{Y\backslash\{2\ell\}}$ for $1 \le \ell \le n$. Finally, for any subset $B \ne A$ of $X$ with $|B|=n-2$, either $2i-1 \in B$ or $2j-1 \in B$, i.e., $x_{2i-1}y_{2j-1}-x_{2j-1}y_{2i-1} \in P_{T'}(G)$, where $T'=B \cup \{b+1~:~b \in B\}.$  So, if we take  $f=(x_{2i-1}y_{2j-1}-x_{2j-1}y_{2i-1})(x_{2i}y_{2j}-x_{2j}y_{2i})$, then  $(J_G:f)=P_{T}(G)$. Hence, $\vn_{P_T(G)}(J_G)\le 4.$
    \end{itemize}
The desired result follows from all the considered cases
    \end{proof}

\begin{remark}\label{rmk}
When $G=C_{3,3}$, using the same reasoning as in the proof of Proposition \ref{prop:Crownv}, we find that $\vn_{P_T(G)}(J_G)\le 4 $ for all $T \in \mathcal{C}(G) \setminus \{X,Y\}$. This implies that Proposition \ref{prop:Crownv} applies to all cases except for certain instances where $T \in \mathcal{C}(G)$. It's worth mentioning that $C_{3,3}$ is a cycle graph with six vertices. \end{remark}

\begin{lemma}\label{lem:vlb}
    Let $G=C_{n,n}$ with $n \ge 3.$ Then $\vn_{P_T(G)}(J_G) \ge 3$ for all $T \in \mathcal{C}(G).$
\end{lemma}
\begin{proof}
It is known from \cite[Theorem 3.20]{vnumberbinomial} that for a connected graph $H$ on the vertex set $[n]$, $\vn(J_H)=1$ if and only if $H$ has a vertex of degree $n-1.$ Since $G$ is a connected graph on $2n$ vertices and $G$ has no vertex of degree $2n-1$, $\vn(J_G) \ge 2. $ 

For any two vertices $u$ and $v$ in the graph $G$, we have $\deg_G(u) = n-1 = \deg_G(v)$. If $\{u,v\} \in E(G)$, then $|N_G(u) \cup N_G(v)| \le 2n-2$, which means that $N_G(u) \cup N_G(v)$ is not equal to $V(G)$. Also, if $\{u,v\} \not\in E(G)$, then either $N_G(u) \cap N_G(v) = \emptyset$ or $|N_G(u) \cap N_G(v)| = n-2$. From Proposition \ref{prop:graph-connectivity}, we conclude that $N_G(u) \cap N_G(v)$ is not a cut set of $G$. Therefore, by \cite[Theorem 4.1]{vconjecture}, we  infer that $\vn(J_G) \neq 2$, and thus, $\vn(J_G) \ge 3$. \end{proof}

Based on Proposition~\ref{prop:Crownv} and Lemma~\ref{lem:vlb}, we deduce that for the crown graph $G = C_{n,n}$ with $n \ge 4$, 
\[
3 \le \vn(J_G) \le \vn_{P_T(G)}(J_G) \le 4 \quad \text{for every } T \in \mathcal{C}(G).
\]
In particular, we have the following result: 
\begin{theorem}\label{thm:vertex_number_bound}
Let $G = C_{n,n}$ with $n \ge 4$. Then, 
\[
3 \le \vn(J_G) \le 4.
\]
\end{theorem}

We strongly believe that $\vn(J_G)=4$. In light of Proposition~\ref{prop:Crownv}, this is equivalent to asserting that $\vn_{P_T(G)}(J_G)=4$ for every $T \in \mathcal{C}(G)$. Indeed, we have already established the case $T=\emptyset$ in Lemma~\ref{lem:mindomset}, where $\vn_{P_{\emptyset}(G)}(J_G)=4$. These observations motivate the following conjecture:

\begin{conjecture}
Let $G=C_{n,n}$ with $n \ge 4$. Then, for every $T \in \mathcal{C}(G)$, we have
\[
\vn_{P_T(G)}(J_G)=4.
\]
\end{conjecture}

To further support the conjecture, we now examine the local $\vn$-numbers of binomial edge ideals of crown graphs. Our approach employs the recursive method introduced by Ohtani \cite{Oh11} to analyze the $\vn$-numbers of both crown graphs and cycles. In \cite{Oh11}, Ohtani established that for any graph $G$ and any internal vertex $v$ of $G$,
\[
J_G = J_{G_v} \cap \Big( (x_v, y_v) + J_{G \setminus v} \Big).
\]
This key technical result was first used by the first author in \cite{K21} and his thesis, and was later generalized to the setting of generalized binomial edge ideals (see \cite{genbin}). Since then, it has become a central tool in the study of binomial edge ideals and has found numerous applications. In the remainder of this section, we showcase a few more applications of this method.

\begin{lemma}\label{lem:v-A}
    Let $G=C_{n,n}$ with $n \ge 3.$ Let $A \subset X$ with $|A|=n-2$.  Then $\vn_{P_T(G)}(J_G) =4$ for  $T=A \cup \{i+1~:~ i \in A\}$.
\end{lemma}
\begin{proof} Due to \cite[Proposition 4.2]{CSTPV20}, it follows that $$\vn_{P_T(G)}(J_G)= \alpha\left( \frac{J_G:P_T(G)}{J_G} \right),$$  here for a finitely generated positively graded module $M$, $\alpha(M)$ denoted the least positive integer $n$ so that $M_n \neq 0$. This means that to calculate $\vn_{P_T(G)}(J_G)$, we need to examine $J_G:P_T(G)$. We present the proof for the case where $A=\{1,3,\ldots, 2n-5\}$ for simplicity, noting that the same proof applies to any $A\subset X$ with $|A|=n-2$, with some changes in notation. According to Theorem \ref{thm:cut-point-prop}, $P_T(G)$ is a minimal prime of $J_G$ and $$P_T(G)=\langle x_i,y_i~:~ i\in T\rangle+\langle x_{2n-3}y_{2n}-x_{2n}y_{2n-3}\rangle + \langle x_{2n-1}y_{2n-2}-x_{2n-2}y_{2n-1}\rangle .$$
    
    According to \cite[Lemma 4.8]{Oh11},  for every internal vertex $v$ of $G$: $$J_G=J_{G_v}\bigcap \left(( x_v,y_v) +J_{G\setminus v}\right).$$ This implies that $J_G:x_v=J_G:y_v=J_{G_v}$ for every $v\in  T.$ As a result,  $(J_G:x_v) \cap (J_G:y_v)=J_{G_v}$ for all $v \in T$. We also know that $x_{2n-3}y_{2n}-x_{2n}y_{2n-3},~   x_{2n-1}y_{2n-2}-x_{2n-2}y_{2n-1} \in J_G$. Thus, we have $J_G:x_{2n-3}y_{2n}-x_{2n}y_{2n-3}=S=J_G:  x_{2n-1}y_{2n-2}-x_{2n-2}y_{2n-1}.$ 

Therefore, we can express $J_G:P_T(G)$ as follows:
\begin{align*} 
J_G:P_T(G) & = \left(\bigcap_{v \in T} (J_G :x_v)  \right)\cap (J_G: x_{2n-3}y_{2n}-x_{2n}y_{2n-3}) \cap (J_G:  x_{2n-1}y_{2n-2}-x_{2n-2}y_{2n-1})\\
& =\left(\bigcap_{v \in T} (J_G :x_v)  \right) \cap S \cap S= \bigcap_{v \in T} J_{G_v}. 
\end{align*} 
    
    Let $f$ be a homogeneous polynomial such that $J_G:f=P_T(G)$. Then, $f \in \left( J_G:P_T(G)\right)$ and $f \notin J_G$ which implies that  $f \in J_{G_v} \setminus J_G$ for $v \in T.$ Thus, for every $v \in T$, $f$ can be expressed as $$\displaystyle f= \sum_{e \in E(G_v)} c_{v,e}f_e,$$  where $c_{v,e}$ are homogeneous polynomials in $S$, and $f_e=x_iy_j-x_jy_i$ if $e=\{i,j\}$. For some $u \in T$,  we may assume that  $$\displaystyle f= \sum_{e \in E(G_u) \setminus E(G)} c_{u,e}f_e.$$ Indeed, if not, then we can replace $f$ by $f'$ where $f=f'+f''$ with $\displaystyle f'= \sum_{e \in E(G_u) \setminus E(G)} c_{u,e}f_e$ and $\displaystyle f''= \sum_{e \in  E(G)} c_{u,e}f_e.$ Without loss of generality, we may assume that $u=1.$ 
    
    We know from Lemma \ref{lem:vlb} that $\deg(f) \ge 3.$ Suppose that $\deg(f)=3.$ Then, every nonzero homogeneous polynomial $c_{v,e}$ appearing in the expression of $f$ is of degree one. Since $f \in J_{G_v} \setminus J_G$, there exists $e_v \in E(G_v) \setminus E(G)$ such that $c_{v,e_v} \neq 0$, i.e., $c_{v,e_v}$ is a degree one homogeneous polynomial. Now, for $u=1$, $f_{e} \in \langle x_i, y_i ~:~ i\in \{2,4\ldots, 2n\}\rangle^2$ for all $e \in E(G_u) \setminus E(G).$ Thus, at most one variable from each terms of $f$  can be in $\langle \{x_i,y_i~:~ i \in \{1,3,\ldots, 2n-1\}\}\rangle.$ This implies that for all $v \in \{a+1~:~a \in A\}$, $$\displaystyle f= \sum_{e \in E(G_v)} c_{v,e}f_e =\sum_{e \in E(G_v)\setminus E(G)} c_{v,e}f_e+\sum_{e \in E(G)} c_{v,e}f_e\not \in J_{G_v} \setminus J_G$$ as $f_{e_v} \in \langle \{x_i,y_i~:~ i \in \{1,3,\ldots, 2n-1\}\}\rangle^2.$ This leads to a contradiction. Thus, $\deg(f) \ge 4,$ and hence, using Proposition \ref{prop:Crownv}, we get that $\vn_{P_T(G)}(J_G) =4. $ 
\end{proof}

\subsection{v-numbers of binomial edge ideals of cycles:}\label{vcycle}
 Deblina et al. in \cite{vconjecture} conjectured that (see \cite[Conjecture 4.13]{vconjecture}) the $\vn$-number of binomial edge ideal of $C_n$ is $\left\lceil\frac{2n}{3}\right\rceil.$ In this subsection, we prove this conjecture.

\begin{theorem}\label{thm: vT-cycle}
Let $G=C_n$ with $n\ge 6$. Then,  $\vn(J_G)= \left\lceil\frac{2n}{3}\right\rceil.$ 
\end{theorem}
\begin{proof}
    Let $T\in\mathcal{C}(G)$ be any subset. If $T$ is empty, then by \cite[Theorem 3.2]{JS23} (also see \cite[Theorem~3.6]{vnumberbinomial}), we know that $\vn_{P_\emptyset(G)}(J_G)=\gamma_c(G)$. It is well-known that $\gamma_c(C_n)=n-2$. Therefore, we have $\vn_{P_\emptyset(G)} (J_G) =n-2 \ge \left\lceil\frac{2n}{3}\right\rceil$. Now, assume that $T$ is not empty, i.e., there exists a vertex $v$ such that $v \in T$. Without loss of generality, let's assume that $v=n$. Let $f\in S$ be a homogeneous polynomial such that $J_G: f = P_T(G)$. According to \cite[Lemma~4.8]{Oh11}, we have $$J_G=J_{G_v}\bigcap \left(( x_v,y_v)+J_{G\setminus v}\right).$$ 
    
    Since $f\in P_A(G)$ for all $A \in \mathcal{C}(G)$ with $A\neq T$, it follows that $f\in P_A(G)$ when $v\notin A$. This means that $f \in J_{G_v}$. It's worth noting that $G_v$ is a graph with vertex set $[n]$ and edge set $E(G_v)=E(G) \cup \{\{1,n-1\}\}$. Thus, we can express $f$ as $f = \sum_{e\in E(G_v)} c_ef_e$, where $c_e$ are homogeneous polynomials in $S$, and $f_e=x_iy_j-x_jy_i$ if $e=\{i,j\}$. We set $e'=\{1,n-1\}$. It is easy to  observe that $f=g+c_{e'}f_{e'}$ where $g \in J_G$, and hence,  $J_G:f=J_G:(c_{e'}f_{e'})$. This leads to $P_T(G)=J_G: c_{e'}f_{e'} =(J_G:f_{e'}):c_{e'}$.
    
    It follows from \cite[Theorem~3.7]{Hilbbino} that
    \begin{align*}\label{eq: vn}
        P_T(G) &= \left(J_{G_{e'}} + \langle x_n, y_n, x_2x_3\cdots x_{n-2}, y_2x_3\cdots x_{n-2}, \ldots, y_2y_3\cdots y_{n-2}\rangle \right) : c_{e'},
    \end{align*}
where $G_{e'}$ is the graph on the vertex set $[n]$ and edge set $E(G_{e'})=E(G) \cup \left\{\{2,n\},n-2,n\} \right\}.$ One can observe that $J_{G_{e'}}+(x_n,y_n)=J_{P_{n-1}} +\langle x_n,y_n \rangle $, where $P_{n-1}$ is the path graph on vertices $1,2,\ldots,n-1$. We set $K = \langle x_2x_3\cdots x_{n-2}, y_2x_3\cdots x_{n-2}, \ldots, y_2y_3\cdots y_{n-2} \rangle$. Then,   
    $P_T(G)= \left(\langle x_n, y_n\rangle + J_{P_{n-1}} + K\right) : c_{e'}.$ 
If $x_n$ or $y_n$ appear in some terms of $c_{e'}$, we can express $c_{e'}$ as $ax_n+by_n+c$ where $a$, $b$, and $c$ are homogeneous polynomials in $S$. Then, we have $P_T(G)=\left(\langle x_n, y_n \rangle + J_{P_{n-1}} + K\right) : c_{e'} =\left(\langle x_n, y_n \rangle + J_{P_{n-1}} + K\right) : c$. Without loss of generality, we assume that $c=c_{e'}$, meaning $x_n$ or $y_n$ do not appear in any term of $c_{e'}$. From \cite[Lemma 3.6]{HJKN}, we have $P_T(G)= \langle x_n, y_n \rangle  + \left(J_{P_{n-1}} + K\right) : c_{e'}$. Let $R=\KK[x_i,y_i~:~ i\in[n-1]]$. Contracting $P_T(G)$ to $R$, we obtain $P_{T\setminus \{n\}}(G\setminus \{n\})= \left(J_{P_{n-1}} + K\right) : c_{e'}.$

Let $A=T\setminus \{n\}.$ Note that $A \in \mathcal{C}(P_{n-1})$ and $A \neq \emptyset.$ Suppose now, to the contrary, $\deg(f) \le \left\lceil\frac{2n}{3}\right\rceil-1.$ Then, $\deg(c) =\deg(f)-2\le \left\lceil\frac{2n}{3}\right\rceil-3.$ Now, for any $i \in A$,  $x_i,y_i \in P_A(P_{n-1})$ which implies  $x_ic, y_ic \in  J_{P_{n-1}} + K$ as $P_A(P_{n-1})= \left(J_{P_{n-1}} + K\right):c$. That is, $$x_ic= \sum_{e \in E(P_{n-1})}u_ef_e+v_1(x_2x_3\cdots x_{n-2})+v_2(y_2x_3\cdots x_{n-2})+\cdots +v_{n-2} (y_2y_3\cdots y_{n-2}),$$ where $u_e$'s and $v_i$'s are homogeneous polynomials in $R$. Since $n\ge 6$,  $\left\lceil\frac{2n}{3}\right\rceil-2 <n-3$. Thus, we get that $x_ic \in J_{P_{n-1}}$ by comparing degrees on both sides. Similarly, we get $y_ic \in J_{P_{n-1}}.$ That is $x_i,y_i \in J_{P_{n-1}}:c.$ Now, for each $i\in A$, it holds that $2 \le i \le n-2$. Thus, we have $K \subseteq J_{P_{n-1}}:c$. Additionally, we also know that $J_{P_{n-1}} \subseteq J_{P_{n-1}}:c$. As a result, $J_{P_{n-1}}+K \subseteq J_{P_{n-1}}:c$, implying that $\left( J_{P_{n-1}}+K \right) :c\subseteq J_{P_{n-1}}:c^2=J_{P_{n-1}}:c $. Therefore, $P_A(P_{n-1})= \left(J_{P_{n-1}} + K\right):c=  J_{P_{n-1}}:c $. Now, due to  \cite[Corollary 3.6]{vconjecture},  $\deg(c) \ge \vn_{P_A(P_{n-1})}(J_{P_{n-1}}) \ge \vn(J_{P_{n-1}}) =\left\lceil\frac{2(n-2)}{3}\right\rceil$. This leads to a contradiction as we assume $ \deg(c) \le \left\lceil\frac{2n}{3}\right\rceil-3.$ Therefore, $\deg(f) \ge \left\lceil\frac{2n}{3}\right\rceil $, and hence, $\vn_{P_T(G)}(J_{G}) \ge \left\lceil\frac{2n}{3}\right\rceil$, which further implies that  $\vn\left(J_{G}\right)\ge \left\lceil\frac{2n}{3}\right\rceil$. Now, the rest follows from  \cite[Corollary~4.12]{vconjecture}.
\end{proof}

As an immediate consequence, we obtain the $\vn$-number of the binomial edge ideal of the crown graph $C_{3,3}.$

\begin{corollary}
    Let $G=C_{3,3}$. Then, $\vn(J_G)=4.$
\end{corollary}
\begin{proof}
    As we pointed out in Remark \ref{rmk} $C_{3,3}$ is a cycle graph on six vertices. The rest follows from Theorem \ref{thm: vT-cycle}. 
\end{proof}

We conclude the article by posing the following question: 

\begin{question}
What are the local $\vn$-numbers of the binomial edge ideal of cycles?\end{question}

\noindent \textbf{Data Availability:} No data is associated with this work.

\printbibliography

%\bibliographystyle{plain}
%\bibliography{references}
%\addbibresource{references.bib} 
\end{document}